\documentclass{amsart}
\usepackage{amssymb,color,hhline}
\usepackage{amsmath}
\usepackage{a4}
 \newtheorem{thm}{Theorem}
 \newtheorem{theorem}[thm]{Theorem}
 
 \newtheorem{lemma}[thm]{Lemma}
 
 \newtheorem{remark}[thm]{Remark}
 \theoremstyle{definition}
 
 \theoremstyle{remark}
 
\makeatletter
\def\mapright#1{\smash{\mathop{\longrightarrow}\limits\sp{#1}}}
\def\qedbox{\hbox{$\rlap{$\sqcap$}\sqcup$}}
\def\SectIntroduction{1}\def\SectTorsionFree{2}\def\SectAffine{3}
\def\SectRep{4}\def\SectWeyl{5}\def\SectRicciFlat{6}
\begin{document}
  \title[Affine curvature operators]{Geometrical representations\\ of equiaffine curvature operators}
\author{P. Gilkey}
  \begin{address}{Mathematics Department, University of Oregon,
 Eugene Oregon 97403 USA}\end{address}
    \begin{email}{gilkey@uoregon.edu}\end{email}
\author{S. Nik\v cevi\'c}
 \begin{address}{Mathematical Institute, SANU,
 Knez Mihailova 35, p.p. 367,
 11001 Belgrade,
 Serbia.}\end{address}
    \begin{email}{stanan@mi.sanu.ac.yu}\end{email}
\begin{abstract}
We examine geometric representability results for various classes of equiaffine curvature operators.
\end{abstract}
 \maketitle
\centerline{This is dedicated to Professor Udo Simon}\medbreak

\medbreak\noindent{\bf \SectIntroduction. Introduction.} Curvature is a fundamental object of study in
differential geometry. It is often convenient to work first in an algebraic context and then pass to the geometric setting subsequently.
In this paper, we shall discuss questions of geometric representability in the affine setting. Here is a brief outline to the paper. In
Section
\SectTorsionFree, we establish notational conventions. In Section
\SectAffine, we show any equiaffine algebraic curvature operator arises from an equiaffine connection. 
In Section \SectRep, we discuss the decomposition of the space of equiaffine algebraic curvature operators into two summands
which are irreducible under the natural action of the general group in dimension $m\ge3$. This gives rise to two additional geometric
representation questions. In Section \SectWeyl, we
discuss the Weyl projective curvature tensor and show any projectively flat equiaffine algebraic curvature operator arises from a
projectively flat equiaffine connection. In Section \SectRicciFlat, we show every Ricci flat equiaffine algebraic curvature operator is
geometrically representable by a Ricci flat equiaffine connection.

Throughout this paper, we will be using results from affine differential geometry; \cite{SSV91} is an excellent
reference for this material which contains a more complete bibliography than we can present here. We
also refer to \cite{BGNS06} for a discussion of
associated representation theory discussed in Section \SectRep. We adopt the {\it Einstein convention} and sum over
repeated indices. All connections are assumed to be torsion free connections on the tangent bundle $TM$ of a smooth manifold $M$. A
connection is equiaffine if it locally admits a parallel volume form or, equivalently (see Lemma \ref{lem-2}), the associated Ricci tensor
is symmetric.

\medbreak\noindent{\bf \SectTorsionFree. Torsion free algebraic curvature operators.} Let
$V$ be a real vector space of dimension
$m$. Let $\mathcal{A}\in\otimes^2V^*\otimes\operatorname{End}(V)$ be an {\it algebraic
curvature operator}, i.e. $\mathcal{A}$ has the symmetries of the curvature operator defined by a torsion free connection:
\begin{equation}\label{eqn-1}\mathcal{A}(x,y)z=-\mathcal{A}(y,x)z\quad\text{and}\quad
\mathcal{A}(x,y)z+\mathcal{A}(y,z)x+\mathcal{A}(z,x)y=0\,.\end{equation}
Let $\mathfrak{A}(V)$ be the vector space of all such operators; see for example the discussion in
\cite{B90}.
In the geometrical setting, let $\nabla$ be a torsion free connection on $TM$. If $P\in M$, let
$\mathcal{R}_P^\nabla$ be the associated curvature operator; $\mathcal{R}_P^\nabla\in\mathfrak{A}(T_PM)$ for any $P\in M$.
We have the following representability theorem:
\begin{theorem}\label{thm-1}
If $\mathcal{A}\in\mathfrak{A}(V)$, then there is a torsion free connection $\nabla$ on
$T(V)$ so that $\mathcal{R}^\nabla_0=\mathcal{A}$.
\end{theorem}

\noindent\textit{Proof.} Although this result is well known, we include the proof as we shall need the construction subsequently. Fix a
basis $\{e_i\}$ for
$V$. Expand
$\mathcal{A}(e_i,e_j)e_k=A_{ijk}{}^\ell e_\ell$. If $v\in V$, expand
$v=x^ie_i$ where $\{x^i\}$ are the dual coordinates on $V$.  Define the Christoffel symbols of a connection
$\nabla$ on
$T(V)$ by setting:
\begin{equation}\label{eqn-2}
\Gamma_{uv}{}^l=\textstyle\frac13(A_{wuv}{}^l+A_{wvu}{}^l)x^w\,.
\end{equation}
Since $\Gamma_{uv}{}^l=\Gamma_{vu}{}^l$, $\nabla$ is torsion free. Since $\Gamma(0)=0$,
we use the curvature symmetries to complete the proof by computing:
\begin{equation}\label{eqn-3}
\begin{array}{l}
\quad\mathcal{R}^\nabla_0(\partial_{x_i},\partial_{x_j})\partial_{x_k}=
\left\{\partial_{x_i}\Gamma_{jk}{}^l-\partial_{x_j}\Gamma_{ik}{}^l\right\}(0)\partial_{x_l}\\
=\frac13\left\{A_{ijk}{}^l+A_{ikj}{}^l-A_{jik}{}^l-A_{jki}{}^l\right\}\partial_{x_l}
   \vphantom{\vrule height11pt}\\
=\frac13\left\{A_{ijk}{}^l-A_{kij}{}^l+A_{ijk}{}^l-A_{jki}{}^l\right\}\partial_{x_l}
=A_{ijk}{}^l\partial_{x_l}\,.\qquad\qedbox\vphantom{\vrule height 11pt}
\end{array}\end{equation}
\medbreak
Theorem \ref{thm-1} shows that every algebraic curvature operator is {\it geometrically representable}. Thus there are
no universal symmetries for the curvature operator of a torsion free connection other than those given above in Equation (\ref{eqn-1}).

\medbreak\noindent{\bf \SectAffine. Equiaffine curvature operators.} Let $\mathcal{A}\in\mathfrak{A}(V)$; the
{\it Ricci tensor} is defined by:
$$
\rho(x,y):=\operatorname{Tr}\{z\rightarrow\mathcal{A}(z,x)y\}\,.$$
We say a curvature operator is an {\it equiaffine curvature operator} or equivalently is a {\it Ricci symmetric
curvature operator} if $\rho(x,y)=\rho(y,x)$ for all $x,y$. We let
$\mathfrak{F}(V)\subset\mathfrak{A}(V)$ be the set of all such operators. Such operators play a central role in many
settings -- see, for example, the discussion in \cite{BDS03, M03, MS99, PSS94}. 

The following geometric observation will play a crucial role in our analysis and justifies the use of the
terminology {\it equiaffine curvature operators}; we refer to
\cite{BGNS06,SSV91} for further details.

\begin{lemma}\label{lem-2}
The following conditions are equivalent for a torsion free connection $\nabla$:
\begin{enumerate}
\item If $\mathcal{O}$ is a coordinate chart, let $\omega=\omega_{\mathcal{O}}:=\Gamma_{ij}{}^jdx^i$. Then $d\omega=0$.
\item One has that $\operatorname{Tr}(\mathcal{R})=0$.
\item The connection $\nabla$ is Ricci symmetric.
\item The connection $\nabla$ locally admits a parallel volume form.
\end{enumerate}
\end{lemma}

\begin{proof} Although this result is well known (see, for example,
\cite{SS-62} page 99), we present the proof as it is elementary and central to our development.
By Equation (\ref{eqn-1}),
\begin{equation}\label{eqn-4}
\operatorname{Tr}\{\mathcal{R}(x,y)-\rho(y,x)+\rho(x,y)\}=0\,.
\end{equation}
The equivalence of Assertions (2) and (3) now follows. 
We have:
$$\begin{array}{l}
R_{ijk}{}^l\partial_{x_\ell}=\nabla_{\partial_{x_i}}\nabla_{\partial_{x_j}}\partial_{x_k}
-\nabla_{\partial_{x_j}}\nabla_{\partial_{x_i}}\partial_{x_k}\\
\quad=\{\partial_{x_i}\Gamma_{jk}{}^\ell-\partial_{x_j}\Gamma_{ik}{}^\ell
+\Gamma_{in}{}^\ell\Gamma_{jk}{}^n-\Gamma_{jn}{}^\ell\Gamma_{ik}{}^n\}\partial_{x_\ell}\,.
  \vphantom{\vrule height 11pt}\end{array}
$$
We show that Assertions (1) and (2) are equivalent by computing:
$$\begin{array}{l}\operatorname{Tr}\{\mathcal{R}_{ij}\}dx^i\wedge dx^j=\{\partial_{x_i}\Gamma_{jk}{}^k-\partial_{x_j}\Gamma_{ik}{}^k
+\Gamma_{in}{}^\ell\Gamma_{j\ell}{}^n-\Gamma_{jn}{}^\ell\Gamma_{i\ell}{}^n\}dx^i\wedge dx^j
  \vphantom{\vrule height 11pt}\\
\quad=\textstyle\{\partial_{x_i}\Gamma_{jk}{}^k-\partial_{x_j}\Gamma_{ik}{}^k\}dx^i\wedge dx^j
  =d\{\Gamma_{ij}{}^jdx^i\}\,.\vphantom{\vrule height11pt}
\end{array}$$
We have that
$$\nabla_{\partial_{x_i}}\{e^{\Phi}dx^1\wedge...\wedge dx^m\}
=\textstyle\{\partial_{x_i}\Phi-\sum_k\Gamma_{ik}{}^k\}\{e^\Phi dx^1\wedge...\wedge dx^m\}\,.
$$
Thus there exists a parallel volume form on $\mathcal{O}$ if and only if $\Gamma_{ik}{}^kdx^i$
is exact. As every closed $1$-form is locally exact, Assertions (1) and (4) are equivalent.
\end{proof}

We have the following representability theorem in this context:
\begin{theorem}\label{thm-3}
Let $\mathcal{A}\in\mathfrak{F}(V)$. Then there exists an equiaffine connection $\nabla$
on $T(V)$ so that $\mathcal{R}^\nabla_0=\mathcal{A}$.
\end{theorem}

\medbreak\noindent\textit{Proof.} Let $\nabla$ be defined by Equation (\ref{eqn-2}); $\mathcal{R}_0^\nabla=\mathcal{A}$. Since
$\mathcal{A}\in\mathfrak{F}(V)$,
$A_{ijk}{}^k=0$ by Equation (\ref{eqn-4}). We use Lemma \ref{lem-2} to show $\nabla$ is equiaffine by computing:
\medbreak\quad
$d\omega=d\Gamma_{ik}{}^kdx^i=\textstyle\frac13(A_{lik}{}^k+A_{lki}{}^k)dx^l\wedge dx^i=-\textstyle\frac13\rho_{li}dx^l\wedge
dx^i=0$.\quad\qedbox
\medbreak\noindent{\bf \SectRep. Representation theory.}
Results of \cite{BGNS06} show that there is a $\operatorname{Gl}(V)$
equivariant short exact sequence:
$$0\rightarrow\ker(\rho)\rightarrow\mathfrak{F}(V)\mapright{\rho}S^2(V^*)\rightarrow0\,.$$
which is $\operatorname{Gl}(V)$ equivariantly split into two irreducible inequivalent $\operatorname{Gl}(V)$ summands. Thus we have:
\begin{equation}
\label{eqn-5}
\mathfrak{F}(V)=S^2(V^*)\oplus\ker(\rho)\quad\text{as a}\quad\operatorname{Gl}(V)\quad\text{module}\,.
\end{equation}
If $m=2$, then $\ker(\rho)=\{0\}$ so we shall assume $m\ge3$ henceforth. The associated projection of $\mathfrak{F}(V)$ on
$\ker(\rho)$ is provided by the {\it Weyl projective curvature operator}
$$\mathcal{P}_{\mathcal{A}}(x,y)z:=\mathcal{A}(x,y)z-\textstyle\frac1{m-1}\{\rho(y,z)x-\rho(x,z)y\}\,.$$
Following \cite{SSV91}, we say that
$\mathcal{A}\in\mathfrak{F}(V)$ is {\it projectively flat} if $\mathcal{P}_{\mathcal{A}}=0$, i.e. if $\mathcal{A}$ belongs to the summand
$S^2(V^*)$. The decomposition of Equation (\ref{eqn-5}) gives rise to two additional natural questions which we shall answer
affirmatively Sections \SectWeyl\ and \SectRicciFlat:
\begin{enumerate}
\item Is
every equiaffine algebraic curvature operator $\mathcal{A}$ which is projectively flat representable by
a equiaffine connection which is projectively flat?
\item Is every Ricci flat algebraic curvature operator representable
by a Ricci flat torsion free connection?
\end{enumerate}

\medbreak\noindent{\bf \SectWeyl. The Weyl projective curvature operator.}
Two connections $\nabla$ and
$\bar\nabla$ are said to be {\it projectively equivalent} if there is a $1$-form $\theta$ so
\begin{equation}\label{eqn-6}
\nabla_xy-\bar\nabla_xy=\theta(x)y+\theta(y)x\,.
\end{equation}
The unparametrized geodesics of $\nabla$ and $\bar\nabla$ coincide if and only if the two connections
are projectively equivalent. Furthermore, if $\nabla$ and $\bar\nabla$ are projectively equivalent, then
they have the same Weyl projective curvature tensors.
A connection is said to be {\it projectively flat} if it is projectively equivalent to a
flat connection or equivalently if $\mathcal{P}_\nabla$ vanishes identically  (Weyl) or if there exist
a local frames $s=(x_1,...,x_m)$ on a neighborhood of any point so that $\nabla_{x_i}x_j=\theta(x_i)x_j+\theta(x_j)x_i$.
\begin{theorem}\label{thm-4}
 Let $\mathcal{A}\in\mathfrak{F}(V)$ be projectively flat. Then there exists an equiaffine projectively flat
connection
$\nabla$ on $T(V)$ so $\mathcal{R}_0^\nabla=\mathcal{A}$.
\end{theorem}

\medbreak\noindent\textit{Proof.} Let $\theta=\theta_idx^i$ be a closed smooth $1$-form on $V$ which vanishes at the origin.
Motivated by Equation (\ref{eqn-6}), we define a projectively flat torsion free connection $\nabla^\theta$ by setting
$\Gamma_{ij}{}^k=\theta_i\delta_j^k+\theta_j\delta_i^k$. Since $\omega:=\Gamma_{ik}{}^kdx^i=(m+1)\theta$, $d\omega=0$ and hence
$\nabla^\theta$ is equiaffine. We compute:
\medbreak\quad $R_{ijk}{}^l(0)=\left\{\partial_{x_i}\Gamma_{jk}{}^l-\partial_{x_j}\Gamma_{ik}{}^l\right\}(0)$
\par\qquad$
=\left\{(\partial_{x_i}\theta_j-\partial_{x_j}\theta_i)\delta_k^\ell+\partial_{x_i}\theta_k\delta_j^l-\partial_{x_j}\theta_k\delta_i^l)
\right\}(0)=\{\partial_{x_k}(\theta_i\delta_j^\ell-\theta_j\delta_i^\ell)\}(0)$
\par\quad 
$\rho_{uv}=(1-m)\partial_{x_u}\theta_v=(1-m)\partial_{x_v}\theta_u$.
\medbreak\noindent Suppose given $\mathcal{A}\in\mathfrak{F}(V)$ which is projectively flat. Let
$f(x):=\frac1{2(1-m)}\rho_{\mathcal{A},ij}x^ix^j$ and let $\theta:=\frac1{1-m}df$. One will then have
$\rho_{\mathcal{R}}(0)=\rho_{\mathcal{A}}$. Since $\nabla^\theta$ and $\mathcal{A}$ are projectively flat, the full tensors agree, i.e.
$\mathcal{R}(0)=\mathcal{A}$ as desired.
\hfill\qedbox

\medbreak\noindent{\bf \SectRicciFlat. Ricci flat curvature operators.}
We say that $\mathcal{A}\in\mathfrak{A}(V)$ is {\it Ricci flat} if $\rho=0$; such an operator is then necessarily equiaffine.
Similarly a torsion free connection $\nabla$ is said to be Ricci flat if $\mathcal{R}_P$ is Ricci flat for all $P\in M$; such a $\nabla$
is then necessarily equiaffine. We have a representability theorem in this final context as well; we may suppose $m\ge3$ since
any Ricci flat connection is flat if $m=2$.

\begin{theorem}\label{thm-5} 
Let $m\ge3$. Let $\mathcal{A}\in\mathfrak{F}(V)$ be Ricci flat. There exists the germ of a real analytic
torsion free Ricci flat connection $\nabla$ on $T(V)$ with
$\mathcal{R}_0^\nabla=\mathcal{A}$.
\end{theorem}

\medbreak\noindent{\bf Proof.} Let $\nabla$ be a torsion free connection. Set
$\omega:=\Gamma_{ij}{}^jdx^i$. If $\omega=0$, then
\begin{equation}\label{eqn-7}
\rho_{jk}=\partial_{x_l}\Gamma_{jk}{}^l-\partial_{x_j}\Gamma_{lk}{}^l+\Gamma_{ln}{}^l\Gamma_{jk}{}^n
-\Gamma_{jn}{}^l\Gamma_{lk}{}^n
=\partial_{x_l}\Gamma_{jk}{}^l-\Gamma_{jn}{}^l\Gamma_{lk}{}^n\,.
\end{equation}
For each pair of
indices $\{i,j\}$, not necessarily distinct, choose an index
$k_{ij}=k_{ji}$ which is distinct from
$i$ and from $j$. If $\Theta$ is a polynomial, let
\begin{equation}\label{eqn-8}
\textstyle(\int_k\Theta)(x):=x^k\int_{0}^1\Theta(x^1,...,x^{k-1},tx^k,x^{k+1},...,x^m)dt
\end{equation}
be the indefinite integral. We then have that
\begin{eqnarray}\label{eqn-9}
&&\partial_{x_k}\textstyle\int_k\Theta=\Theta,\\
&&\textstyle|\Theta(x)|\le C|x|^\mu\text{ for }|x|\le\varepsilon
\quad\Rightarrow\quad|\int_k\Theta(x)|\le C|x|^{\mu+1}\text{ for }|x|\le\varepsilon\,.\label{eqn-10}
\end{eqnarray}

 We
suppose
$\nabla$ is defined by the connection
$1$-form
\begin{equation}\label{eqn-11}
\Gamma=\Gamma_1+\Gamma_3+...+\Gamma_{2\nu-1}+...
\end{equation}
where the $\Gamma_{2\nu-1}$ and $\Theta_{2\nu}$
are the polynomials defined recursively by the relations:
\begin{eqnarray}
&&\Gamma_{1,ij}{}^l=\textstyle\frac13(A_{kij}{}^l+A_{kji}{}^l)x^k,\label{eqn-12}\\
&&\Theta_{2\nu,ij}=\Gamma_{2\nu-1,in}{}^l\Gamma_{2\nu-1,jl}{}^n\label{eqn-13}\\
&&\qquad+\textstyle\sum_{\mu<\nu}\{\Gamma_{2\nu-1,in}{}^l\Gamma_{2\mu-1,jl}{}^n
+\Gamma_{2\mu-1,in}{}^l\Gamma_{2\nu-1,jl}{}^n\},\nonumber\\
&&\Gamma_{2\nu+1,ij}{}^l=\delta_{k_{ij}}^l\textstyle\int_{k_{ij}}\Theta_{2\nu,ij}\quad\text{for}\quad\nu\ge1\,.\label{eqn-14}
\end{eqnarray}
 Since $\rho_{\mathcal{A}}=0$, we have $\Gamma_{1,ij}{}^j=0$. As $k_{ij}$ is an index distinct from $i$ and $j$, it
follows that
$\Gamma_{2\nu+1,ij}{}^j=0$ for $\nu\ge1$. Thus 
$\omega=0$.
We have $\Gamma_{1,ij}{}^l=\Gamma_{1,ji}{}^l$. It is clear from Equation (\ref{eqn-13}) that $\Theta_{2\nu,ij}=\Theta_{2\nu,ji}$;
we integrate to see 
$$\Gamma_{2\nu+1,ij}{}^l=\Gamma_{2\nu+1,ji}{}^l\quad\text{for}\quad\nu\ge1\,.$$
Thus $\nabla$ is torsion free. Clearly $\Gamma_1=O(|x|)$. We use induction, the recursion relations of Equations (\ref{eqn-13}) and
(\ref{eqn-14}), and the estimate of Equation (\ref{eqn-10}) to see
\begin{equation}\label{eqn-15}
||\Gamma_{2\nu-1}||=O(|x|^{2\nu-1})\quad\text{for}\quad\nu\ge1
\end{equation}
where $||$ is a suitably chosen operator norm. As $\Gamma=\Gamma_1+O(|x|^3)$, Equation (\ref{eqn-3})
shows $R_{ijk}{}^l(0)=A_{ijk}{}^l$. Since $\partial_l\Gamma_{2\nu+1,ij}^l=\Theta_{2\nu,ij}$ and since $\rho(0)=0$, Equation
(\ref{eqn-7}) implies that $\rho\equiv0$.

We must improve Equation (\ref{eqn-15}) to show that there exist $C>0$ and $\varepsilon>0$ so
\begin{equation}\label{eqn-16}
||\Gamma_{2\nu-1}||\le C^\nu|x|^{2\nu-1}\quad\text{for}\quad|x|\le\varepsilon\,.
\end{equation}
We complexify and permit $x^i\in\mathbb{C}$ to be coordinates on $V\otimes_{\mathbb{R}}\mathbb{C}$.
We extend Equation (\ref{eqn-8}) to the complex domain; Equations (\ref{eqn-9}) and (\ref{eqn-10})
continue to hold. We then extend Equations (\ref{eqn-12}), (\ref{eqn-13}), and (\ref{eqn-14}) to the complex
domain as well.

Since $\Gamma_1$ is a homogeneous linear polynomial, $||\Gamma_1||\le C_1|x|$ where we may take $C_1\ge1$. Set $C=4C_1$ and
$\varepsilon=(8C_1)^{-1}<1$. We suppose inductively that the estimate in Equation (\ref{eqn-16}) holds for
$\nu\le\mu$ and attempt to establish the estimate for $\Gamma_{2\nu+1}$. We use the recursive definition of Equation (\ref{eqn-13}) to estimate
\begin{eqnarray*}
&&||\Theta_{2\nu}||\le 2||\Gamma_{2\nu-1}||\cdot\textstyle\sum_{\mu\le\nu}||\Gamma_{2\mu-1}||\\
&\le&2(4C_1)^{\nu}|x|^{2\nu-1}\left\{C_1|x|+\textstyle\sum_{2\le\mu\le\nu}(4C_1)^\mu|x|^{2\mu-1}\right\}\\
&\le&2(4C_1)^{\nu}|x|^{2\nu-1}\left\{C_1|x|+(4C_1)^2|x|^3(1-4C_1|x|^2)^{-1}\right\}\\
&\le&2(4C_1)^{\nu}|x|^{2\nu-1}\left\{C_1|x|+|x|\right\}\le(4C_1)^{\nu+1}|x|^{2\nu}\,.
\end{eqnarray*}
Integrating this estimate then establishes the estimate of Equation (\ref{eqn-16}) since we pick up an extra power of $|x|$ in
the decay by Equation (\ref{eqn-10}). This
establishes the convergence in the $C^0$ sup norm of the series defining $\Gamma$ in Equation (\ref{eqn-11}) for $|x|\le\varepsilon$. The
$\Gamma_i$ are polynomials and hence holomorphic. Recall that the uniform limit of holomorphic functions is holomorphic and furthermore
the convergence of the series in question is in fact uniform in the
$C^k$ norm for any $k$ for $|x|\le\frac12\varepsilon$. Thus $\Gamma$ is smooth and the convergence is in the $C^k$ norm.\hfill\qedbox

\begin{remark}\label{rmk-6}
\rm It is in fact necessary to correct the construction of Theorem \ref{thm-1} in this setting. For example, suppose that
the non-zero components of
$\mathcal{A}$ are given by:
$$A_{211}{}^2=1,\quad A_{121}{}^2=-1,\quad A_{311}{}^3=-1,\quad A_{131}{}^3=1\,.$$
Then $\mathcal{A}$ is Ricci flat.
If we use Equation (\ref{eqn-2}) to define $\nabla$, then $\rho_{\nabla,22}=\frac29x_2^2\ne0$.
\end{remark}

\begin{remark}\label{rmk-7}
\rm In contrast to the constructions performed in previous sections, the construction used to prove Theorem \ref{thm-5} is
not $\operatorname{Gl}(V)$ equivariant.
\end{remark}


\medbreak\noindent{\bf Acknowledgments.} Research of P. Gilkey partially supported by the
Max Planck Institute in the Mathematical Sciences (Leipzig) and by Project MTM2006-01432 (Spain). 
Research of S. Nik\v cevi\'c partially supported by Project 144032 (Srbija). Both authors acknowledge with
gratitude many conversations with Udo Simon; he is our mentor in this area and we have been honored by his friendship.

\end{document}